\documentclass[graybox]{svmult}

\usepackage{mathptmx}       
\usepackage{helvet}         
\usepackage{courier}        
\usepackage{type1cm}        
\usepackage{makeidx}         
\usepackage{graphicx}        
\usepackage{multicol}        
\usepackage[bottom]{footmisc}

\usepackage[psamsfonts]{amssymb}
\usepackage{amsfonts,amsmath,euscript,wrapfig,url}

\newcommand{\rr}{\mathbb{R}}

\newcommand {\rd} {\mathbb{R}^d}

 \newcommand {\al} {\alpha}

\newcommand {\ga} {\gamma}

\newcommand {\la} {\lambda}

\newcommand {\te} {\theta}
\newcommand {\fy} {\varphi}

\newcommand{\ep}{\varepsilon}
\newcommand{\UU}{{\bigcup\limits}}

\newcommand {\ra} {\rightarrow}
\newcommand{\IN}{{\subset}}

\newcommand{\NI}{{\supset}}

\newcommand {\mmm}{{\setminus}}

\newcommand{\8}{{\infty}}

\newcommand{\ia}{{I^*}}

\newcommand{\bj}{{\bf {j}}}

\newcommand{\bi}{{\bf {i}}}
\newcommand{\bk}{{\bf {k}}}

\newcommand{\wP}{{\widetilde P}}

\newcommand{\eS}{{\EuScript S}}

\newcommand{\eC}{{\EuScript C}}

\newcommand{\eZ}{{\EuScript Z}}

\def \Lip {\mathop{\rm Lip}\nolimits}
\def \min {\mathop{\rm min}\nolimits}

\begin{document}

\title*{On Dendrites Generated By Symmetric Polygonal Systems: The Case of Regular Polygons}
\titlerunning{On Symmetric Dendrites} 
\author{Mary Samuel, Dmitry Mekhontsev and Andrey Tetenov \thanks{Supported by Russian Foundation of Basic Research projects 16-01-00414 and 18-501-51021}}
 \authorrunning{M.Samuel, D.Mekhontsev, A.Tetenov} 
\institute{ Mary Samuel\at  Department of Mathematics,
Bharata Mata College, Kochi, India, \email{marysamuel2000@gmail.com}
\and  Dmitry Mekhontsev\at Sobolev Mathematics Institute,
Novosibirsk, Russia, \email{mekhontsev@gmail.com}\and  Andrei Tetenov\at Gorno-Altaisk State University 
and
Novosibirsk State University and
Sobolev Mathematics Institute,
Novosibirsk, Russia, \email{atet@mail.ru}}
%
%
\maketitle

\abstract*{We  define   $G$-symmetric  polygonal systems of similarities and study the properties of symmetric dendrites, which appear as their attractors. This allows us to find the conditions under which the attractor of a zipper becomes a dendrite.}

\abstract{We  define   $G$-symmetric  polygonal systems of similarities and study the properties of symmetric dendrites, which appear as their attractors. This allows us to find the conditions under which the attractor of a zipper becomes a dendrite.}

\section{Introduction}
\label{sec:1}
Though the study of topological properties of dendrites from the viewpoint of general topology proceed for more than  three quarters of a century \cite{Char,Kur, Nad}, the attempts to study the geometrical properties of self-similar dendrites are rather fragmentary.

In 1985, M.~Hata \cite{Hata}  studied the connectedness properties of self-similar sets and proved that if a dendrite is an attractor of a system of weak contractions in a complete metric space, then the set of its endpoints is infinite. In 1990 Ch.~Bandt showed in his unpublished paper \cite{BS} that the Jordan arcs connecting pairs of points of a post-critically finite self-similar dendrite are self-similar, and the set of possible values for dimensions of such arcs is finite. Jun Kigami in his work \cite{Kig95} applied the methods of harmonic calculus on fractals to dendrites; on a way to this he developed effective approaches to the study of  structure of self-similar dendrites. D.Croydon in his thesis \cite{C} obtained heat kernel estimates for continuum random tree and for certain family of p.c.f. random dendrites on the plane.

In our recent works \cite{STV1,STV2,STV3} we considered systems $\eS$ of contraction similarities in $\rd$
defined by some polyhedron $P\IN\rd$, which we called contractible   $P$-polyhedral systems.
We proved that the attractor of such system $\eS$ is a dendrite $K$ in $\rd$;  we showed that the orders of points $x\in K$ have an upper bound, depending only on $P$; and that
Hausdorff dimension of the set $CP(K)$ of the cut points of $K$ is strictly smaller than the dimension of the set $EP(K)$ of its end points unless $K$ is a Jordan arc.

Now we extend our approach to the case of symmetric P-polygonal systems $\eS$ and show that the symmetric dendrites $K$ which are the attractors of these systems, have clear and obvious structure: their main tree is a symmetric n-pod (Proposition \ref{npod}),
all the vertices of the polygon $P$ are the end points of $K$ and show that for $n>5$ each ramification point of $K$ has the order n (Proposition \ref{comp1}). We show that the augmented system $\widetilde\eS$ contain subsystems $\eZ$ which are zippers whose attractors are subdendrites of the dendrite $K$ (Theorem \ref{zipdend}). 

\subsection{Dendrites}
\label{subsec:1}

\begin{definition}
  A {\em dendrite} is a locally connected continuum containing no simple closed curve. 
\end{definition}

In the case of dendrites the order $Ord(p,X)$ of the point  $p$  with respect to  $X$ is equal to the number  of components of the set $X \setminus \{p\}$.  
{Points of order 1 in a continuum $X$ are called {\em end points} of $X$;   the set of all end points of $X$   will
be  denoted by $EP(X)$. A point $p$ of a continuum $X$ is called a {\em cut point} of $X$ provided that $X \setminus \{p\}$ is
not connected; the set of all cut points of $X$ will be denoted by $CP(X)$. 
Points of order at least 3 are called {\em ramification points} of $X$; the  set of all ramification points of $X$ will be denoted by $RP(X)$. }

According to \cite[Theorem 1.1]{Char}, for a continuum $X$ the following conditions are equivalent:
 $X$ is dendrite;
 every two distinct points of $X$ are separated by a third point;
 each point of $X$ is either a cut point or an end point of $X$;
 each nondegenerate subcontinuum of $X$ contains uncountably many cut points of $X$;
 the intersection of every two connected subsets of X is connected; X is locally connected and uniquely arcwise connected.

\subsection{ Self-similar sets}
\label{subsec:1}

Let $(X, d)$ be a complete metric space. 
A mapping $F: X \to X$ is a contraction if $\Lip F < 1$.\\ 
The mapping $S: X \ra X$ is called a similarity if $ d(S(x), S(y)) = r d(x, y) $ for all $x, y\in X$ and some fixed r. 

\begin{definition} 
Let $\eS=\{S_1, S_2, \ldots, S_m\}$ be a system of   contraction maps on a complete metric space $(X, d)$.
 A nonempty compact set $K\IN X$ is the attractor of the system  $\eS$, if $K = \bigcup \limits_{i = 1}^m S_i (K)$. \end{definition}
 
 The system $\eS$ defines the Hutchinson operator $T$  by the equation $T(A) = \bigcup \limits_{i = 1}^m S_i (A)$. By Hutchinson's Theorem, the attractor $K$ is uniquely defined by $\eS$ and for any compact set $A\IN X$ the sequence $T^n(A)$ converges to $K$.
 
 We also call the   subset $K \IN X$ self-similar with respect to $\eS$. Throughout the whole paper, the maps $S_i\in \eS$ are supposed to be  similarities and the set $X$ to be $\rr^2$.
 
 {\bf Notation.}
 $I=\{1,2,...,m\}$ is the set of indices, $\ia=\bigcup\limits_{n=1}^\8 I^n$  
is the set of all finite $I$-tuples, or multiindices $\bj=j_1j_2...j_n$.  By $\bi\bj$ we denote the concatenation of the corresponding multiindices;\\ 
we say $\bi\sqsubset\bj$, if $\bi=i_1\ldots i_n$ is the initial segment in $\bj=j_1\ldots j_{n+k}$ or $\bj=\bi\bk$ for some $\bk\in\ia$;\\
if $\bi\not\sqsubset\bj$ and $\bj\not\sqsubset\bi$, $\bi$ and $\bj$ are {\em incomparable};\\
we write
$S_\bj=S_{j_1j_2...j_n}=S_{j_1}S_{j_2}...S_{j_n}$ 
and for the set $A
\subset X$ we denote $S_\bj(A)$ by $A_\bj$; \\
we also denote by $G_\eS=\{S_\bj, \bj\in\ia\}$ the semigroup, generated by $\eS$.\\
The set of all infinite sequences $I^{\8}=\{{\bf \al}=\al_1\al_2\ldots,\ \ \al_i\in I\}$ is called the
{\em index space}; and $\pi:I^{\8}\rightarrow K$ is the {\em index map
 }, which sends a sequence $\bf\al$ to  the point $\bigcap\limits_{n=1}^\8 K_{\al_1\ldots\al_n}$.

\subsection{Zippers}
\label{subsec:1}
 The simplest  way to construct  a self-similar  curve is  to  take a polygonal line and then replace each of its segments by a smaller copy of the same polygonal line; this construction is called   zipper and was studied in \cite{ATK,Tet06}:
 
\begin{definition} Let $X$  be a complete metric space. A system $\eS = \{S_1, \ldots, S_m\}$
of contraction mappings of $X$ to itself is called a {\em zipper} with vertices $\{z_0, \ldots, z_m\}$
and signature $\ep = (\ep_1, \ldots, \ep_m)$, $\ep_i \in\{0,1\}$, if for    $ i = 1, \ldots m$, $S_i (z_0) = z_{i-1+\ep_i}$ and $S_i (z_m) = z_{i-{\ep_i}}$.\end{definition}

A zipper $\eS$ is a {\em Jordan zipper} if and only if one (and
hence every) of the structural parametrizations of its attractor establishes a homeomorphism of the interval
$J = [0, 1]$ onto $K(\eS)$.

\begin{theorem}\label{Jordan} Let $\eS = \{S_1, . . . , S_m\}$ be a zipper with vertices $\{z_0, . . . , z_m\}$ in a complete metric
space $X$ such that all contractions $S_j : X\to X$ are injective. If for arbitrary $i, j \in I$ the set
$K_i\cap K_j$ is empty for $|i-j| > 1$ and is a singleton for $|i-j| = 1$ then every structural parametrization
$\fy: [0, 1]\to K(\eS)$ of $K(\eS)$ is a homeomorphism and $K(\eS)$ is a Jordan arc with endpoints $z_0$ and $z_m$.\end{theorem}

%
%

\section{Contractible $P$-polygonal   systems}
\label{sec:2}

Let $P$ be a convex polygon in ${\rr}^2$ and $V_P=\{A_1, \ldots, A_{n_P}\}$  be the set of its vertices, where $n_P=\#V_P$.\\
Consider a system of contracting similarities $\eS = \{S_1, \ldots, S_m\}$, which possesses  the following properties:\\
{\bf(D1)}\   For any $k \in I$, the set $P_k = S_k (P)$ is contained in $P$;
\\
{\bf(D2)}\    For any $i\neq j$,\ $i, j \in I$, $P_i \bigcap P_j$ is either empty or  a common vertex of $P_i$ and $P_j$;\\
  {\bf(D3)}\  For any  $A_k\in V_P$ there is a map $S_i\in\eS$ and a vertex $A_l\in V_P$ such that $S_i(A_l)= A_k$;\\
{\bf(D4)}\   The set    ${\wP} = \bigcup \limits_{i = 1}^m P_i$ is contractible.

\begin{definition}\label{pts}
The system $(P,\eS)$ satisfying the conditions {\bf D1-D4} is called  a contractible P-polygonal  system of similarities.
\end{definition}

Applying Hutchinson operator  $T (A)=\bigcup\limits_{i\in I} S_i(A)$ of the system $\eS$ to the polygon $P$, we get the set ${\wP}^{(1)} = \bigcup \limits_{i\in I} P_i$. Taking the iterations of $T$, we define   ${\wP}^{(n + 1)} = T(\wP^{(n)})$  and get a nested family of contractible compact sets
${\wP}^{(1)}\NI {\wP}^{(2)}\NI \ldots \NI  {\wP}^{(n)}\NI\ldots$.
By Hutchinson's theorem, the intersection of this nested sequence is the attractor $K$ of the system $\eS$.

The following Theorem was proved by the authors in \cite{STV1,STV2,STV3}:

\begin{theorem}\label{main}   Let $\eS$ be a contractible P-polygonal system, and let $K$ be its attractor. Then $K$ is a dendrite.\end{theorem}

Since $K$ is a dendrite, for any vertices $A_i,A_j \in V_P$ there is an unique Jordan arc ${\ga}_{ij}\IN K$ connecting $A_i,A_j$.
The set  $\hat\ga=\bigcup\limits_{i\neq j}\ga_{ij}$ is a subcontinuum of the dendrite  $K$, all of whose end points   are contained in $V_P$, so $\hat\ga$ is a finite dendrite or topological tree  \cite[{\bf A.17}]{Char}.

\begin{definition}\label{defmt}
The union $\hat\ga=\bigcup\limits_{i\neq j}\ga_{ij}$ is called 
{\em the main tree} of the dendrite $K$. The ramification points of
$\hat\ga$ are called {\em main ramification points} of the dendrite
$K$.
\end{definition}

We consider $\hat\ga$ as a topological graph whose vertex set $V_{\hat\ga}$ is the union of $V_P$ and the set of ramification points $RP(\hat\ga)$, while the edges of $\hat\ga$ are the  components of $\hat\ga\mmm V_{\hat\ga}$.  

The following Proposition \cite{STV1} show the relation between the vertices of $P$ and end points, cut points and ramification points of $\hat\ga$.

\begin{proposition}\label{comp} 
a) For any $x\in\hat\ga$, $\hat\ga=\bigcup\limits_{ j=1}^n\ga_{A_jx}$.\\
b) $A_i$ is a cut point of $\hat\ga$, if there are $j_1,j_2$ such that $\ga_{j_1i}\cap\ga_{j_2i}=\{A_i\}$;\\
c) the only end points of $\hat\ga$ are the vertices $A_j$ such that $A_j\notin CP(\hat\ga)$;\\
d) if $\#\pi^{-1}(A_i)=1$, then $Ord(A_i,K)\le n-1$,
otherwise $$Ord(A_i,K)\le (n-1)\left(\left\lceil{\dfrac{\te_{max}}{\te_{min}}}\right\rceil-1\right),$$ 
where $\te_{max},\te_{min}$ are maximal and minimal values of vertex angles of $P$.
 \end{proposition}

It was proved in \cite{STV1,STV2,STV3} that each cut point of the dendrite $K$ is contained in some image $S_\bj(\hat\ga)$ of the main 
tree: 

\begin{theorem}\label{order} The following statements are true:\\
a) $CP(K)\IN\bigcup\limits_{\bj\in I^*}S_\bj(\hat\ga)$.\\ b) For each cut point $y\in K\mmm\bigcup\limits_{\bj\in I^*}S_\bj(V_P)$, $\#\pi^{-1}(y)=1$ and there is $S_\bi$  and $x\in\hat\ga$, such that
$y=S_\bi(x)$ and   $Ord(y,K)=Ord(x,{\hat\ga})$.\\ 
c) If $y\in\bigcup\limits_{\bj\in I^*}S_\bj(V_P)$ and $\#\pi^{-1}(y)=s$, there are 
multiindices $\bi_k, k=1,..,s$ and vertices $x_1,...,x_s$, 
such that for any  $k$, $S_{\bi_k}(x_k)=y$ and for any $l\neq k$,  $S_{\bi_k}(P)\cap S_{\bi_l}(P)=\{y\}$\\ and
 $ Ord(y,K)=\sum\limits_{k=1}^s Ord(x_k,{\hat\ga})\le(n_P-1)\left(\left\lceil{\dfrac{2\pi}{\te_{min}}}\right\rceil-1\right)$.
\end{theorem}

Moreover, the dimension of the set of the end points is always greater then the one of the set of cut points \cite{STV2,STV3}

\begin{theorem}\label{equal}
Let $(P,\eS)$ be a contractible $P$-polyhedral   system and $K$ be its attractor.
(i) $\dim_H(CP(K))=\dim_H(\hat\ga)\le \dim_H EP(K)=\dim_H(K)$;
(ii) $\dim_H(CP(K))=\dim_H(K)$ iff $K$ is a Jordan arc.
\end{theorem}

\section{Symmetric polygonal  systems}
\label{sec:3}


\begin{definition} \label{sps} Let $P$ be a polygon and $G$ be a non-trivial symmetry group of the polygon $P$.  
Let $\eS$ be a contractible $P$-polygonal  system such that for any $g \in G$ and any $S_i\in\eS$,  there are such  $g' \in G$ and  $S_j\in\eS$ that 
$g\cdot S_i=S_j\cdot g'$.  Then the system of mappings $\eS = \{S_i ,  i = 1,  2,  \ldots, m\}$ is called a contractible { $G$-symmetric} 
$P$-polygonal  system. \end{definition}

\begin{center}  \includegraphics[width=.23 \textwidth]{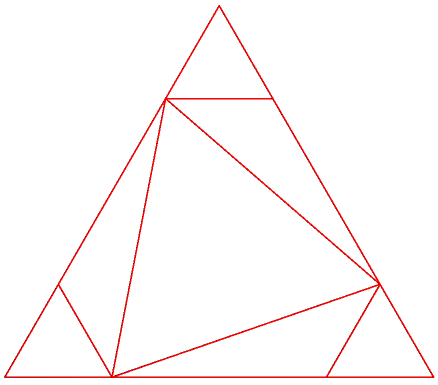}\quad \includegraphics[width=.22 \textwidth]{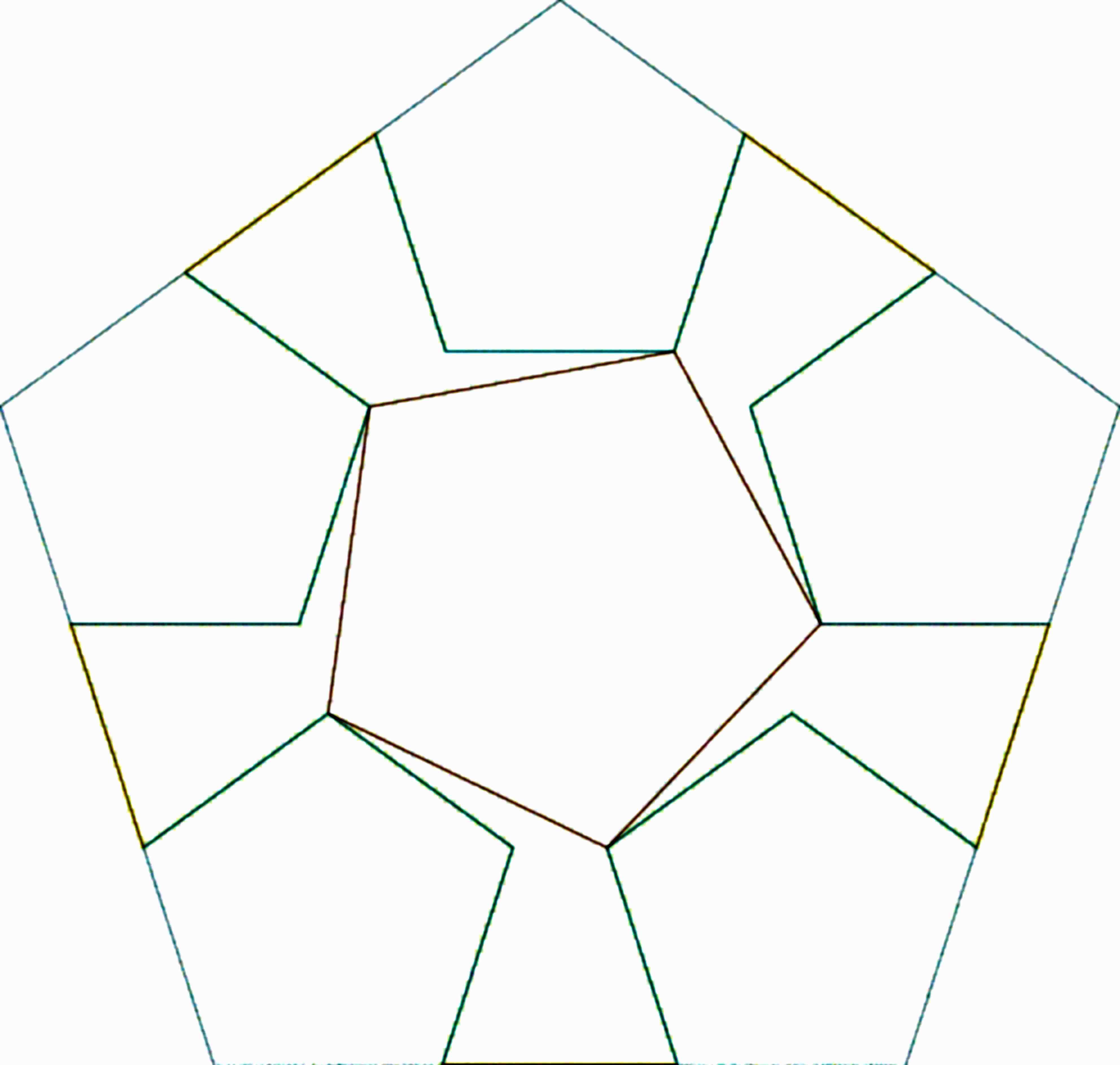}\quad\includegraphics[width=.18\textwidth]{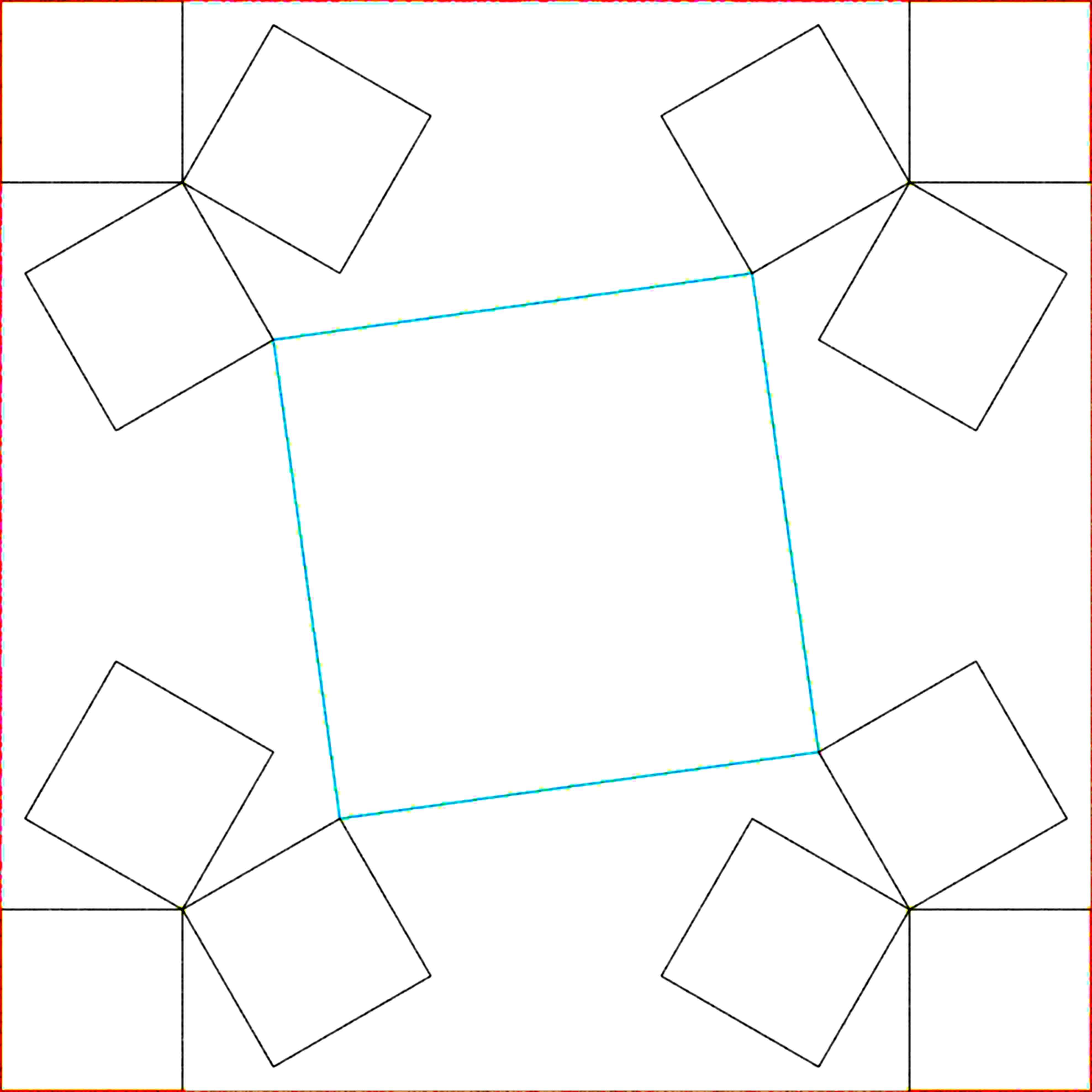}\quad \includegraphics[width=.2 \textwidth] {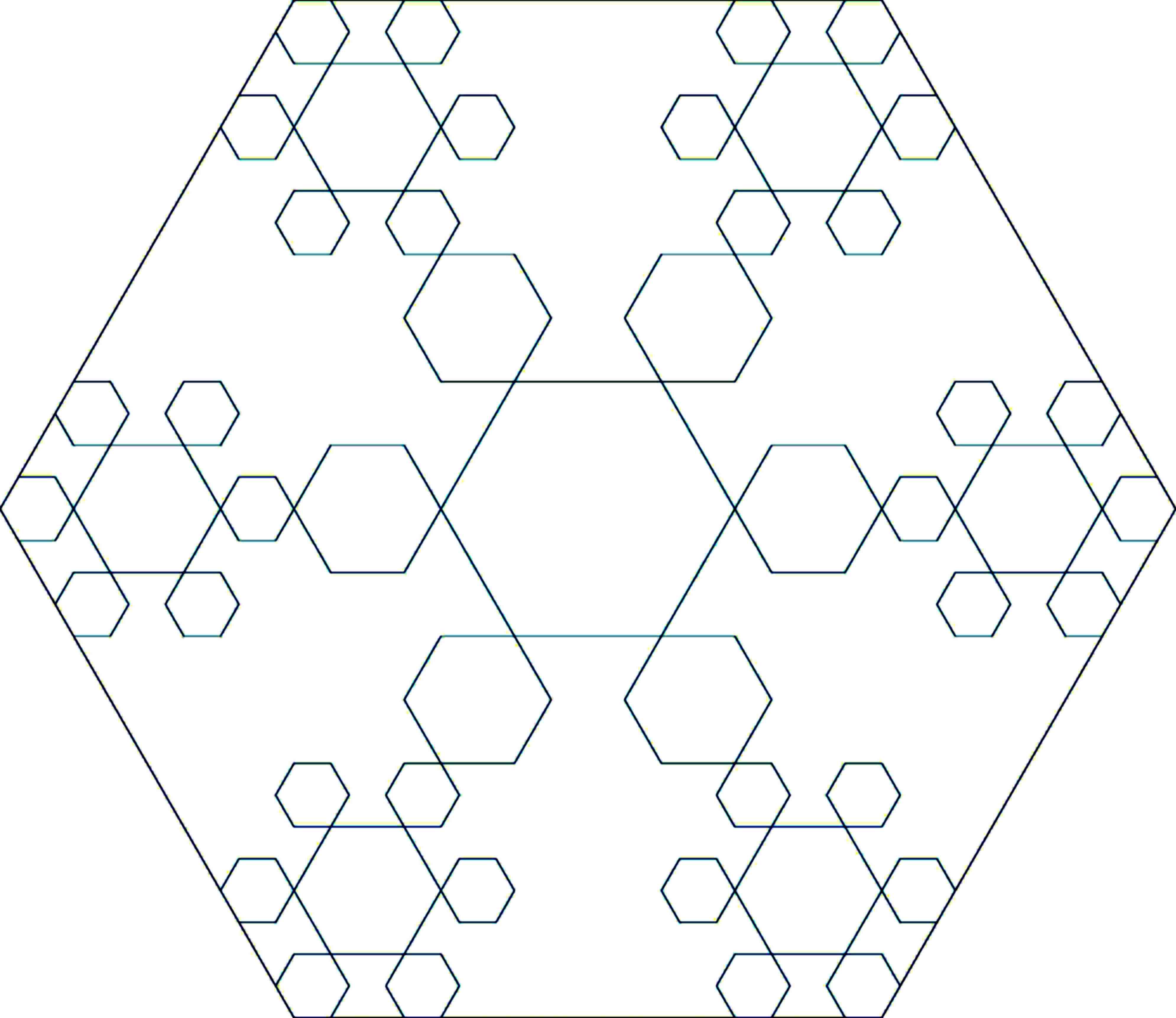}
\end{center}

For convenience we will call such systems symmetric polygonal systems or SPS, if this does not cause ambiguity in choice of $P$ and $G$.

\begin{theorem} \label{main1}
The attractor $K$ of a { symmetric} polygonal  system  and its main tree $\hat\ga$ are symmetric with respect to the group $G$.
\end{theorem}
\begin{proof} Let $\eS=\{S_1,\ldots, S_m\}$. Take $g\in G$. The map $g^*:\eS \to \eS$, sending each $S_i$ to respective $S_j$ is a permutation of $\eS$, therefore
 $g(\UU_{i=1}^m S_i(P))=\UU_{i=1}^m S_i(P)$, or $g(\wP)=\wP$.

Moreover, it follows from the Definition \ref{sps} that for any $\bi=i_1 \ldots i_k$ there is such
$\bj=j_1 \ldots j_k$  and such $g'\in G$ that  $g\cdot S_\bi=S_\bj\cdot g'$. Therefore for any $g$, $g(\wP^k)=\wP^k$. Since $K=\bigcap\limits_{k=1}^\8 \wP^k$,
$g(K)=K$.  Since $g$ preserves the set of vertices of $P$, $g(\hat\ga)=\hat\ga$.\end{proof}
\begin{center} \includegraphics[width=.22 \textwidth]{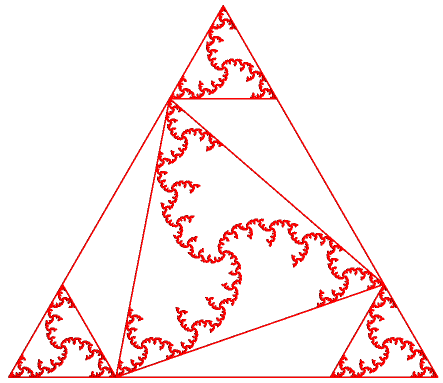}\quad \includegraphics[width=.22 \textwidth] {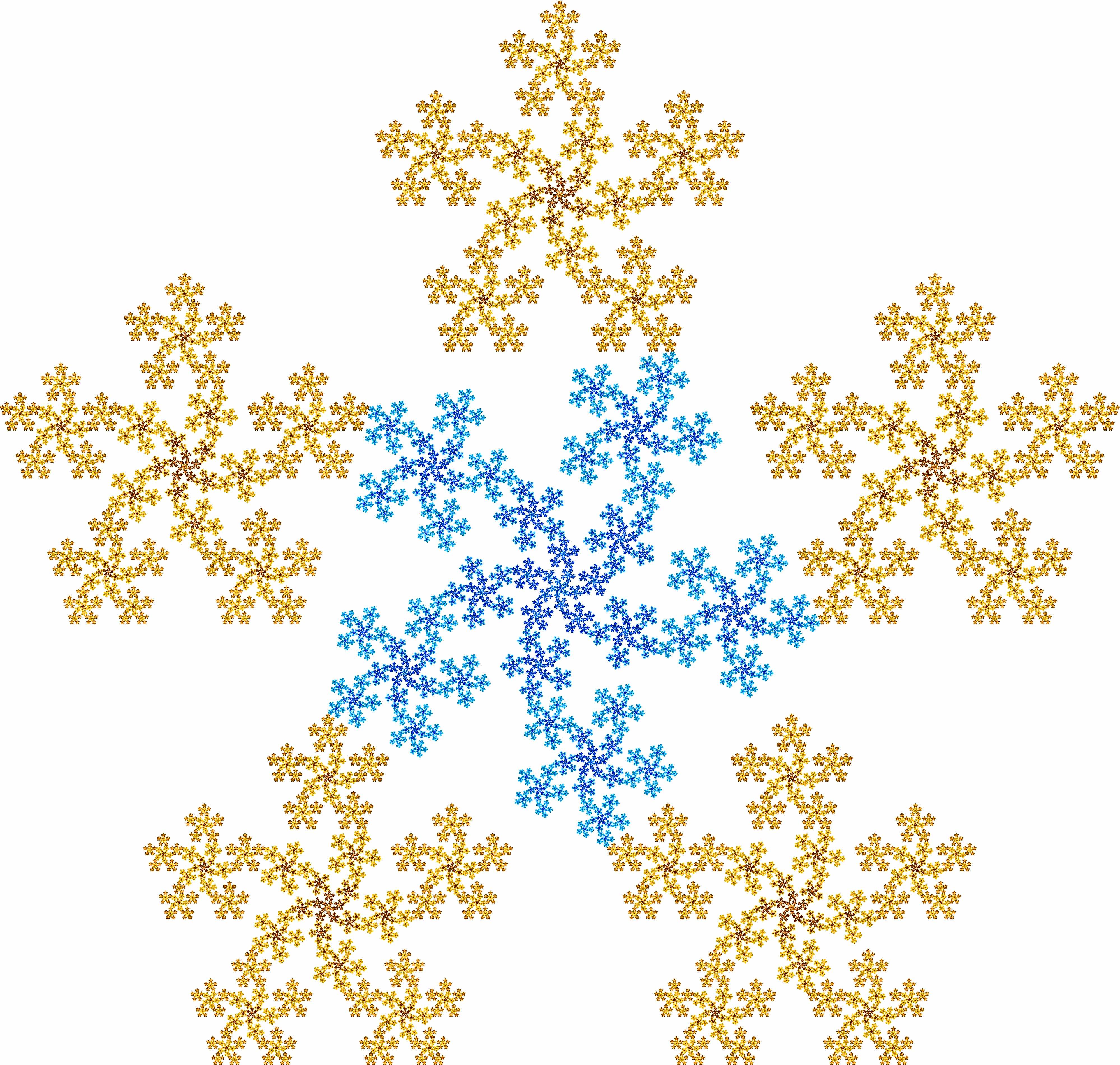}\quad \includegraphics[width=.18 \textwidth] {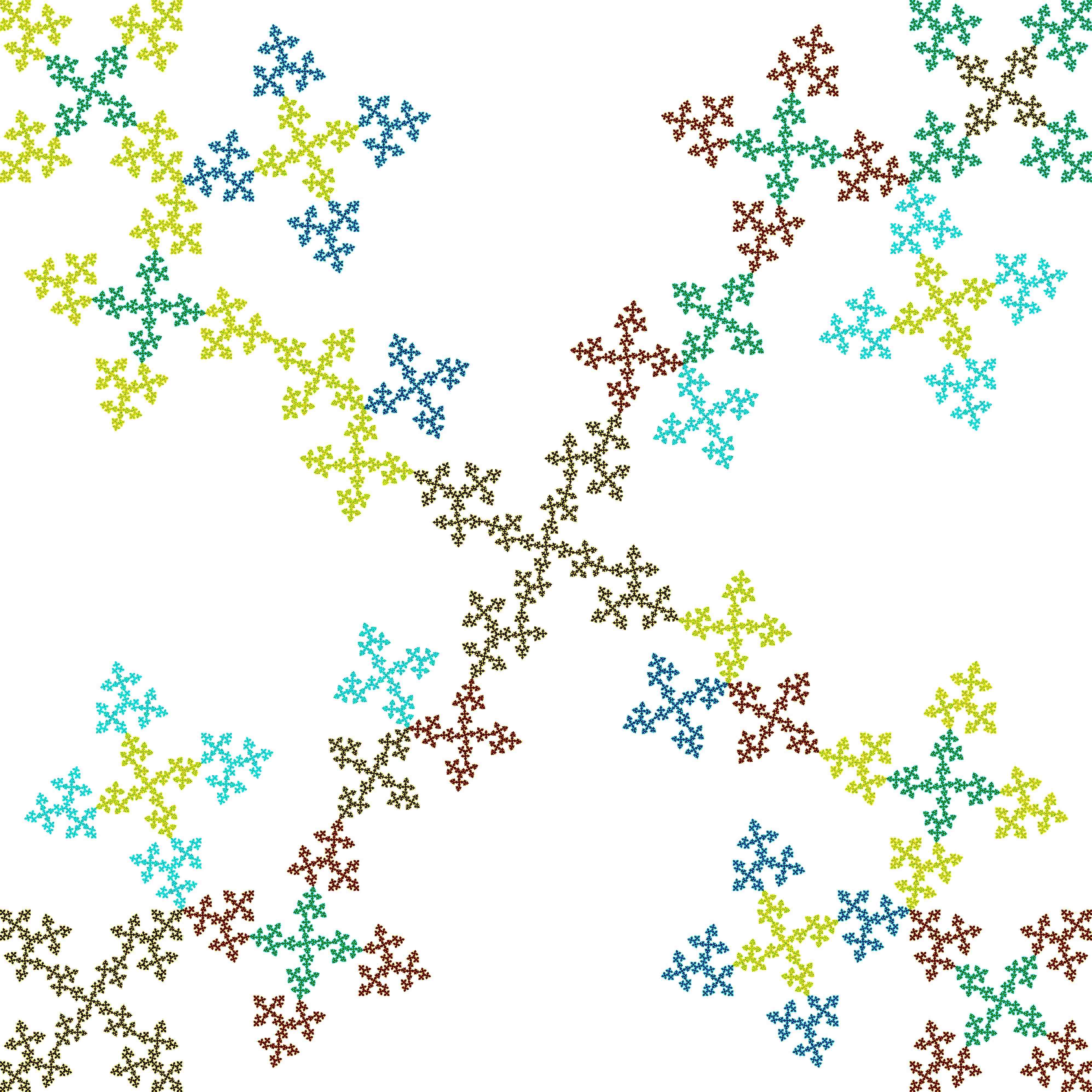}\quad \includegraphics[width=.22 \textwidth] {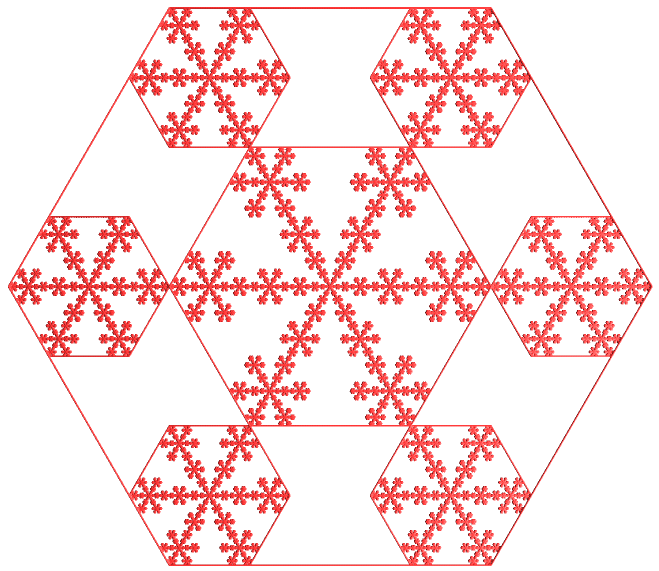}
\end{center}

\begin{corollary} If $\eS$ is a {$G$-symmetric} $P$-polygonal  system then  $\eS^{(n)}=\{S_\bj, \bj\in I^n\}$ is a {$G$-symmetric} $P$-polygonal system.\end{corollary}

\begin{corollary} Suppose $\eS=\{S_1,\ldots, S_m\}$ is a {$G$-symmetric} $P$-polygonal  system, $K$ is the attractor of $\eS$,
 $g_1,\ldots,g_m\in G$ and  $\eS'=\{S_1g_1,\ldots,S_mg_m\}$. Then $K$ is the attractor of the system $\eS'$.
\end{corollary}

\begin{proof}Let $K'$ be the attractor of the system $\eS'$ and put $\wP' = \UU_{i=1}^m (S_i \circ g_i (P))$ . 
Observe that for any $i$, $g_i (P)=P$, therefore $\wP'=\wP$ and ${\wP}^{'(k)}={\wP}^{(k)}$.
Then $K'=\bigcap\limits_{k=1}^\8 \wP^{'(k)} =K$.\end{proof}

\begin{definition} Let $\eS=\{S_1,\ldots, S_m\}$ be a {$G$-symmetric} $P$-polygonal  system. The system $ \widetilde\eS=\{S_i\cdot g, S_i\in \eS, g\in G\}$ is called the {\em augmented system} for $\eS$.\end{definition}
The system $ \widetilde\eS$   has the same attractor $K$ as $\eS$  and generates the augmented semigroup $G(\widetilde\eS)$ consisting of all maps of the form $S_\bj\circ g_i$, where $g_i\in G$. 


\subsection{The case of regular polygons}
\label{subsec:3}

\begin{proposition}\label{npod} Let $P$ be a regular n-gon and $G$ be the rotation group of $P$. Then the center $O$ of $P$ is the only ramification point of the main tree and $Ord(O,\hat\ga)=n$.
\end{proposition}

\begin{center}
\includegraphics[width=.22 \textwidth]{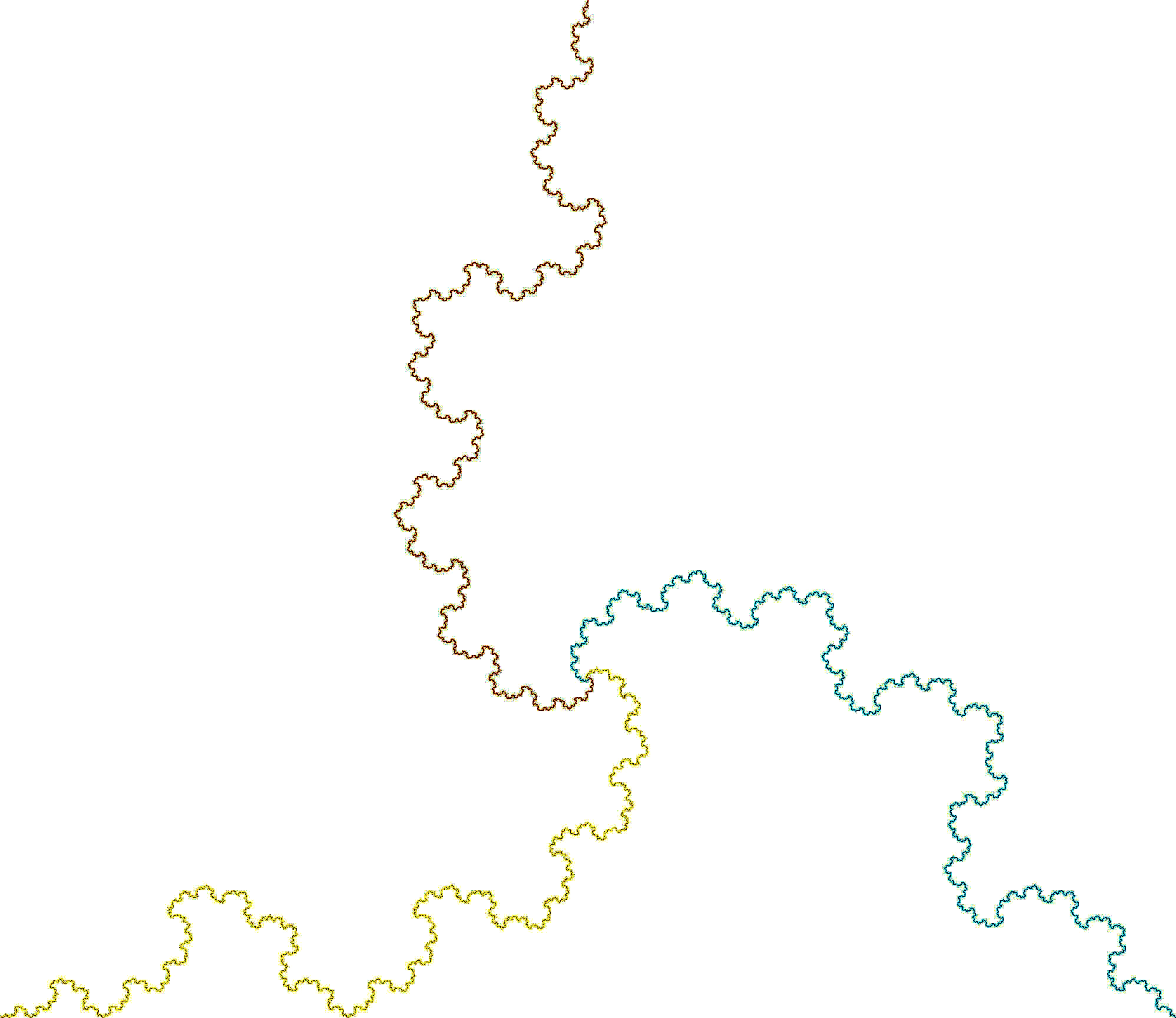}\quad\includegraphics[width=.22 \textwidth]{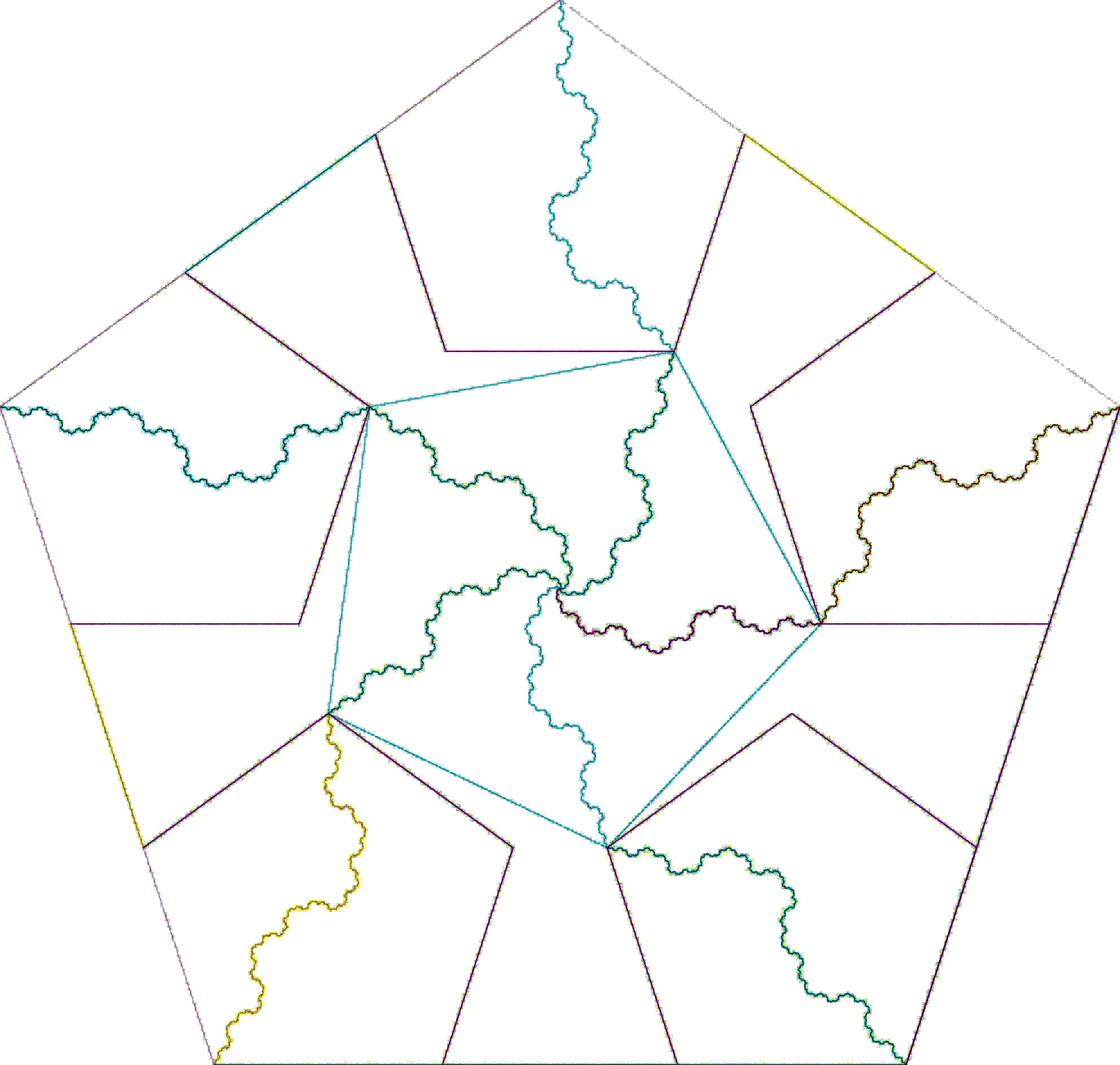}\quad \includegraphics[width=.22 \textwidth]{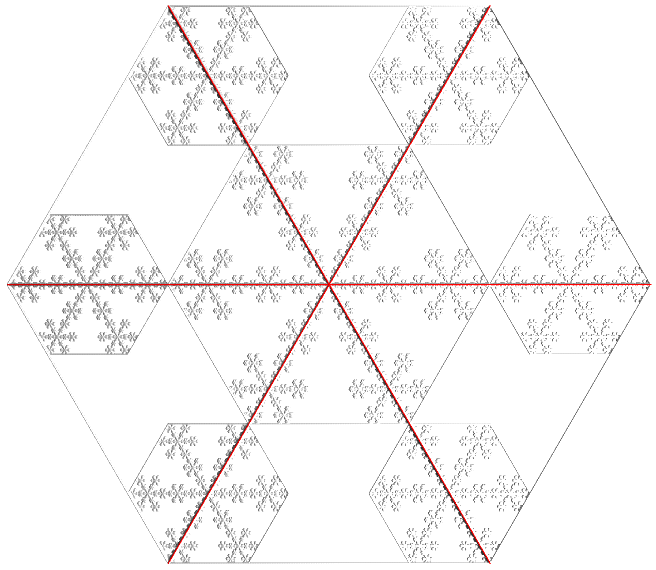}
\end{center}

\begin{proof}
Consider the main tree $\hat\ga$. It is a fine finite system \cite{Kig95}, which is invariant with respect to  $G$. Let $f$ be the rotation of $P$ in the angle $2\pi/n.$\\

Suppose $V$ and $E$ be the number of vertices and edges respectively of the main tree. For any  edge  $\lambda
\subset \hat\ga$, $f(\lambda)\cap\lambda$ is either empty, 
or is equal to $\{O\}$, and in the latter  case $O$ is the endpoint of both $\la$ and $f(\la)$. 
In each case all the edges $f^k(\la)$ are different.
 Therefore $E$ is a multiple of $n$. \\

If $A'$ is a vertex of $\hat\ga$ and $A'\neq O$, then all the points  $f^k(A'), k=1,...,n$ are different, so the number of vertices of $\hat\ga$, different from $O$, is also a multiple of $n$.\\

Since $\hat\ga$ is a tree, $V = E +1$. Therefore the set of vertices contains $O$, which is the only 
invariant point for  $f $.   Denote the unique subarc of $\hat\ga$ with endpoints $O$ and  $A_k$ by $
\ga_k$.  Then for any $k=1,...,n$,  $\ga_k=f^k(\ga_n)$. By Proposition \ref{comp}  $\bigcup \limits _{k=1} ^n \ga_k = \hat\ga$. Thus the center  $O$ is  the only 
ramification point of $ \hat\ga$ and $Ord(O, \hat\ga) = n$.\end{proof}

\begin{corollary}
All vertices of the polygon $P$ are the end points of the main tree.
\end{corollary}

\begin{proof} For any $k=1,...,n$ there is  an unique arc $\ga_k$ of the main tree meeting the vertex $A_k$ of the polygon $P$, so $Ord(A_k, \hat\ga) = 1$ by Proposition \ref{comp}. 

Since all the vertex angles of $P$ are equal, for each vertex $A_k$ of $P$, there is unique $S_k\in\eS$ such that 
$P_k=S_k(P)\ni A_k$, so $\#\pi^{-1}(A_k)=1$ and   by Theorem \ref{order}, $Ord(A_k, K)=Ord(A_k, \hat\ga) = 1$.

Then all vertices of the polygon $P$ are the end points of the main tree as well as of the dendrite $K$.\end{proof}

\begin{lemma}\label{gaon}
Each arc $\ga_{k}$ is the attractor of a Jordan zipper.
 \end{lemma}

\begin{proof} We prove the statement for the arc $\ga_n$, because for 
$\ga_k=f^k(\ga_n)$ it follows automatically.

 If $n>3$,
there is a similarity $S_0\in\eS$, whose fixed point is $O$. Indeed, there is some $S_0\in\eS$ for which $P_0=S_0(P)\ni O$. The point $O$ cannot be the vertex of $P_0$, 
otherwise  the polygons $f(P_0)$ and $P_0$ would intersect more than in one point. Therefore $f(P_0)=P_0$ and 
$S_0(O)=O$.

Observe that for any two vertices $A_i,A_j$ of $P$, the arc $\ga_{A_iA_j}$ is the union
$\ga_i\cup\ga_j$.

There is an unique chain of subpolygons $P_{l_k}=S_{l_k}(P), k=0, \ldots,s$ connecting $P_0$ and $P_n$ and 
containing $\ga_n$, where  $S_ {l_0}=S_0$ and  $S_ {l_s}=S_n$.
For each $k=1,\ldots,s$, there are $i_k$ and $j_k$ such that  $\ga_n\cap P_{l_k}=S_{l_k}(f^{i_k}(\ga_n)\cup f^{j_k}(\ga_n))$. 
Therefore $$\ga_n=\bigcup\limits_{k=1}^sS_{l_k}\left(f^{i_k}(\ga_n)\cup f^{j_k}(\ga_n)\right)\cup S_0(\ga_n).$$
The arcs on the right hand satisfy the conditions of Theorem  \ref{Jordan}, therefore the system $$\{S_0,S_{l_1}f^{i_1},S_{l_1}f^{j_1},...,S_{l_s}f^{i_s},S_{l_s}f^{j_s}\}$$ is a Jordan zipper whose attractor is  a Jordan arc with endpoints $O$ and $A_n$.

If $n=3$, it is possible that for some $l_1$, $O$ is a vertex of a triangle $S_{l_1}(P)$ and there is an unique chain of subpolygons $P_{l_k}=S_{l_k}(P), k=1, \ldots,s$, where $S_ {l_s}=S_3$. Repeating the same argument, we get a system
$\{S_{l_1}f^{i_1},S_{l_1}f^{j_1},...,S_{l_s}f^{i_s},S_{l_s}f^{j_s}\}$ is a Jordan zipper whose attractor is  a Jordan arc with endpoints $O$ and $A_3$.\end{proof}

\begin{corollary} If $P$ is a regular  n-gon and the symmetry group $G$ of the system $\eS$ is the dihedral group $D_n$ then
$\ga_{OA_i}$ is the line segment and the set of cut points of $K$ has dimension 1.
\end{corollary} 

\begin{proof} 
Since $D_n$ contains a symmetry with respect to the straight line containing $O$ and $A_n$,$\ga_n$ itself is a straight line segment.\end{proof}

%
%

From the above statements we see that   Proposition \ref{comp} and Theorem \ref{order} in the case of $G$-symmetric polygonal systems with $G$ being the rotation group of order $n$ and $P$ a regular n-gon, acquire the following form:

\begin{proposition}\label{comp1} Let $\eS$ be a $G$-symmetric $P$-polygonal system of similarities, where $P$ is a regular n-gon 
and $G$ contains the rotation group of $P$. Then:\\
a) $V_p\IN EP(\hat\ga)\IN EP(K)$;\\
b) For each cut point $y\in K\mmm\bigcup\limits_{\bj\in I^*}S_\bj(V_P)$, either $y=S_\bi(O)$  for some
$\bi\in I^*$ and  $Ord(y,K)=n$, or $Ord(y,K)=2$.\\ 
c) For any $y\in\bigcup\limits_{\bj\in I^*}S_\bj(V_P)$ there is unique $x\in \bigcup\limits_{i\in I}S_i(V_P)$, that 
 $$Ord(y,K)=Ord(x,K)=\#\pi^{-1}(y)=\#\pi^{-1}(x)=\#\{i\in I: x \in S_i(V_P)\} \le 1+\left\lceil{\dfrac{4}{n-2}}\right\rceil$$
 \end{proposition}
 \begin{proof}All vertex angles of $P$ are $\te=\pi-\dfrac{2\pi}{n}$, so 
$ \left\lceil{\dfrac{2\pi}{\te_{min}}}\right\rceil-1=1+\left\lceil{\dfrac{4}{n-2}}\right\rceil$.

a) Take a vertex $A_i\in V_P$. There is unique $j\in I$ such that $A_i\in S_j(V_P)$. For that reason $\#\pi^{-1}(A_i)=1$.
Since $S_j(P)$ cannot contain the center $O$, $\#(S_j(V_P)\cap\hat\ga)=2$, therefore by Theorem \ref{comp}, $Ord(A_i,\hat\ga)=1$
and  $Ord(A_i,K)=1$, so $A_i\in EP(K)$.

b) If for some $\bj\in I^*$,   $y=S_\bj(O)$, then $Ord(y,K)=n$.  Since for any point $x\in CP(\hat\ga)\mmm\{O\}$,  $Ord(x,\hat\ga)=2$,  the same is true for $y=S_\bj(x)$ for any $y\in I^*$.

c) Now let $\eC=\{C_1,...,C_N\}$ be the full collection of those points $C_k\in\bigcup\limits_{i\in I}S_i(V_P)$ for which 
$s_k:=\#\{j\in I: S_j(V_P)\ni C_k\}\ge 3$. By Theorem \ref{comp}, $\#\pi^{-1}(C_k)=s_k$ and $Ord(C_k,K)=s_k$, while $s_k\le 1+\left\lceil{\dfrac{4}{n-2}}\right\rceil$\\

Then, if $y\in S_\bj(C_k)$ for some $\bj\in I^*$ and $C_k\in\eC$, then $\#\pi^{-1}(y)=s_k=Ord(y,K)$.
\end{proof}

Applying the Proposition \ref{comp1} to different n, we get the possible ramification orders for regular n-gons:\\
1. If $n \ge 6$ then all ramification points of $K$ are the images $S_\bj(O)$ of the centre $O$ and have the order $n$.\\
 2. If $n = 4$ or $5$ then there is a finite set of ramification points $x_1,...,x_r$, whose order is equal to $3$ such that each $x_k$ is a common vertex of polygons $S_{k1}(P), S_{k2}(P), S_{k3}(P)$. Then each ramification point is represented either as $S_\bj(O)$ and has the order $n$ or as $S_\bj(x_k)$ and has the order 3.\\
 3. If $n = 3$ the centre is a ramification point of order 3 and those ramification points which are not images of $O$ will have an order less than or equal to $5$. 

\subsection{Self-similar zippers, whose attractors are dendrites}
\label{subsec:3}

\begin{theorem}\label{zipdend}

Let $(\eS, P)$  be a $G$-symmetric $P$-polygonal system  of similarities. Let $A,B$ be two vertices of the polygon $P$ and $L$ be 
the line segment $[A,B]$. If $\eZ=\{S_1', \ldots, S_k'\}$ is such family of maps from $\widetilde\eS$ that $\tilde L=
\UU_{i=1}^k S_i'(L)$ is a polygonal line connecting $A$ and $B$, then the attractor $K_\eZ$ of $\eZ$ is a 
subcontinuum of $K$. If for some subpolygon $P_j$, $\widetilde L\cap P_j$  contains more than one segment, then 
$K_\eZ$ is a dendrite. 
\end{theorem}

\begin{center}
\includegraphics[width=.3 \textwidth]{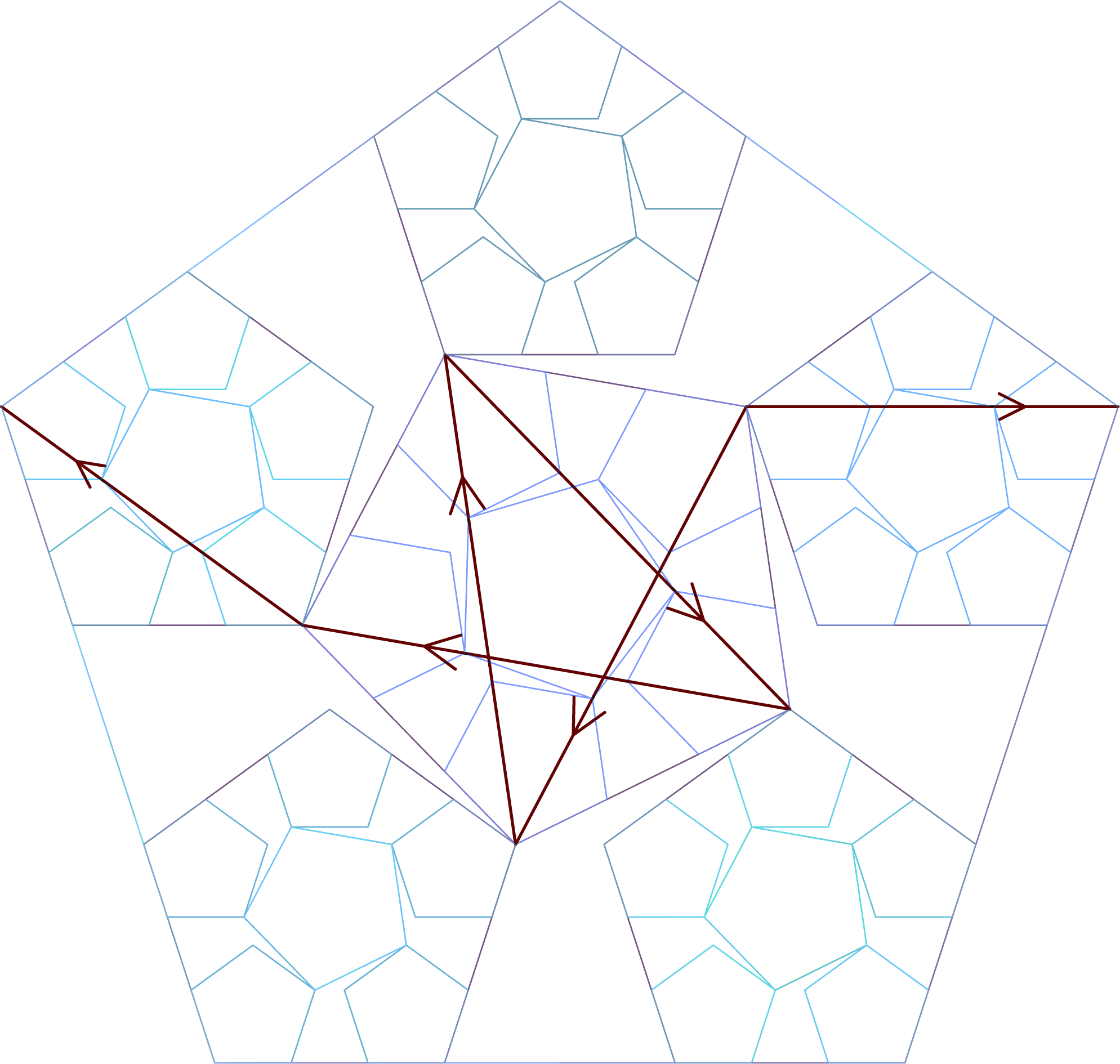} \qquad \qquad \includegraphics[width=.3 \textwidth]{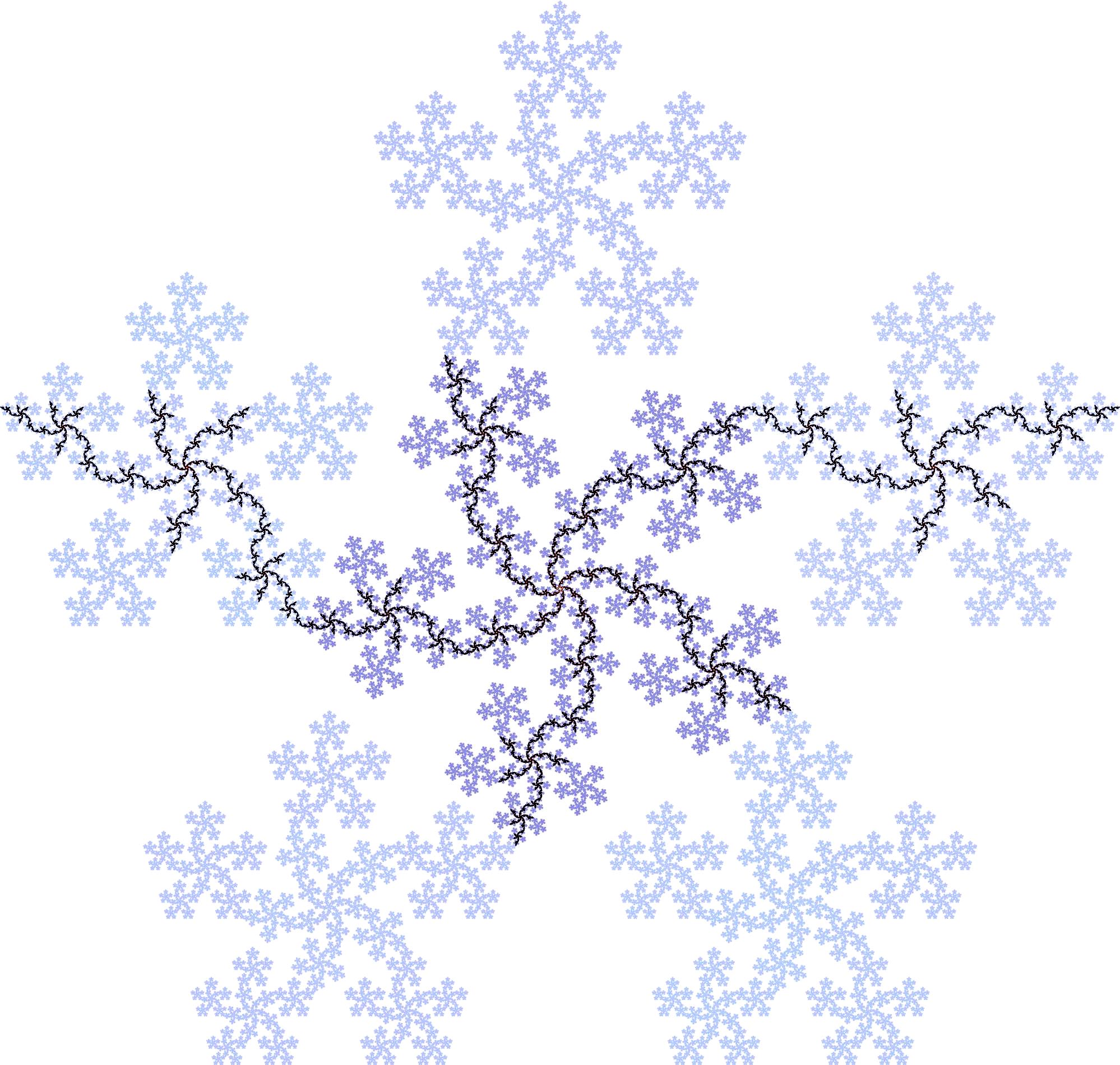}\\ {\tiny The G-SPS $(\eS, P)$  with polygonal lines $\widetilde L$ joning the vertices $A, B$ and attractors of $\eS$ and the zipper $\eZ$.}\end{center}

\begin{proof} 
Since $\eZ\IN\widetilde\eS$, the attractor $K_\eZ$ is a subset of $K$. The system $\eZ$ is a zipper with 
vertices $A,B$, therefore $K_\eZ$ is a continuum, and therefore  is a subdendrite of 
the dendrite $K$. Let $\ga_{AB}$ be the Jordan arc  connecting $A$ and $B$ in $K_\eZ$, and,  therefore, in $K$. 
By  the proof of Lemma \ref{gaon}, $\ga_{AB}=\ga_{OA}\cup\ga_{OB}$. If the maps $S_{i_1}', S_{i_2}'$ send $L$ to two 
segments belonging to the same subpolygon $P_{i_0}$, then $S_{i_1}'(\ga_{AB})\bigcup  S_{i_2}'(\ga_{AB})$ is equal 
to $S_{i_1}'(\ga_{OA}\bigcup \ga_{OB})\bigcup  S_{i_2}'(\ga_{OA}\bigcup\ga_{OB})$.
The  set $\{S_{i_1}'(A),S_{i_1}'(B), S_{i_2}'(A), S_{i_2}'(B)\} $ contains at least 3 different points, therefore 
$S_{i_1}'(O)$ is a ramification point of $K_\eZ$ of order at least 3.
\end{proof}


\begin{corollary}
Let $u_i$ be the number of segments of the intersection $\tilde L\cap P_i$ and $u=\max u_i$. Then maximal  order 
of ramification points of $K_\eZ$ is greater or equal to $\min(u+1,n)$.
\end{corollary}

\begin{proof}  Suppose $\widetilde L\bigcap P_i$ contains $u$ segments of $\widetilde L$. Then the set contains at least $u+1$ vertices of $P_i$ if $u<n$ and contains  $n$ vertices of $P_i$ if  $u= n-1$ or $n$. Then  the set $K_\eZ\cap P_i$ contains at 
least $u+1$ (resp. exactly $n$) different images of the arc $\ga_{OA}$.\end{proof}


\begin{thebibliography}{99.}%
%
%
\bibitem{ATK} Aseev, V.~V., Tetenov, A.~V., Kravchenko, A.~S.: On Self-Similar Jordan Curves on the Plane. Sib. Math. J. \textbf{44}(3), 379---386 (2003).


\bibitem{BS} Bandt, C., Stahnke, J.: Self-similar sets 6. Interior distance on deterministic fractals.
preprint, Greifswald 1990.


\bibitem{Char}
  Charatonik, J., Charatonik, W.: Dendrites. Aport. Math. Comun.
\textbf{22} 227---253(1998).

\bibitem {C} Croydon, D.: Random fractal dendrites, Ph.D. thesis. St. Cross College, University of Oxford, Trinity(2006)


\bibitem{Hata}
Hata, M.: On the structure of self-similar sets. Japan.J.Appl.Math.\textbf{3}, 381---414.(1985)

\bibitem{Kig95} Kigami, J.: Harmonic calculus on limits of networks and its application to dendrites. J. Funct. Anal. \textbf{128}(1) 48---86, (1995)



\bibitem{Kur}
Kuratowski, K.:  Topology. Vol. 1 and 2. Academic Press and PWN, New York(1966)



\bibitem{Nad}Nadler,~S.~B.,~Jr.: Continuum theory: an introduction. M. Dekker (1992)




\bibitem{STV1}
Samuel, M., Tetenov, A.~ V.~, Vaulin,~D.A.: Self-Similar Dendrites Generated by Polygonal Systems in the Plane. Sib. Electron. Math. Rep. \textbf{14},  737---751(2017)

\bibitem{Tet06}
Tetenov, A.~V.: Self-similar Jordan arcs and graph-oriented IFS. Sib. Math.J. \textbf{47}(5), 1147---1153 (2006).


\bibitem{STV2}
Tetenov, A.~V., Samuel, M., Vaulin, D.A.:
On dendrites generated by polyhedral systems and their ramification points.  Proc.  Krasovskii Inst. Math. Mech. UB RAS \textbf{23}(4), 281---291 (2017) DOI: 10.21538/0134-4889-2017-23-4-281-291 (in Russian)

\bibitem{STV3}
Tetenov A.~V., Samuel, M., Vaulin, D.A.:
On dendrites, generated by polyhedral systems and their ramification points. arXiv:1707.02875v1 [math.MG], 7 Jul 2017.






\end{thebibliography}
\end{document}